\documentclass[12pt]{amsart}

\usepackage{amsmath}
\usepackage{amscd}
\usepackage{amssymb}
\usepackage{latexsym}
\usepackage{mathrsfs}

\parskip=\medskipamount

\newtheorem*{Thm*}{Theorem}

\theoremstyle{definition}

\newtheorem{theorem}{Theorem}[section]

\newtheorem{proposition}[theorem]{Proposition}

\newcommand{\utimes}{\kern0.05em\buildrel{\times}\over{\rule{0em}{0.004em}}\kern-0.9em\cup \kern0.2em}
\newcommand{\sutimes}{\mathrel{\kern0em\buildrel{\mathsf{x}}\over{\rule{0em}{0.0em}} \kern-0.35em\cup\kern-0.0em}}

\allowdisplaybreaks[1]

\title[Hausdorff Continuity]{On the Hausdorff Continuity of Free L\`{e}vy Processes and Free Convolution Semigroups.}
\author[John D. Williams]{John D. Williams$^{\dagger}$}
\thanks{\footnotesize $^{\dagger}$Research supported in part by the Alexander von Humboldt-Stiftung and the AMS-Simons Foundation.}
\date{\today}
\address{John D. Williams, Universit\"{a}t des Saarlandes,
Fachrichtung Mathematik.
Postfach 151150.
66041, Saarbr\"{u}cken, Germany.}
\email{williams@math.uni-sb.de}

\begin{document}
\maketitle

\begin{abstract}
Let $\mu$ denote a Borel probability measure and let $\{ \mu_{t} \}_{t\geq 1}$ denote the free additive convolution semigroup of Nica and Speicher.  We show that the support of these measures varies continuously in the Hausdorff metric for $t >1$.
We utilize complex analytic methods and, in particular, a characterization of the absolutely continuous portion of these supports due to Huang.
\end{abstract}

\section{Introduction}

Let  $\mu$ and $\nu$ denote Borel probability measures on $\mathbb{R}$.  We refer to $supp(\mu)$ as the support of this measure.
We denote by $\mu \boxplus \nu$ as their \textit{free additive convolution} \cite{V1} and refer to \cite{DNV} for an introduction to this theory.

The \textit{free convolution semigroup} associated to $\mu$ is the  family of Borel probability measures $\{ \mu_{t} \}_{t\geq 1}$
satisfying \begin{enumerate}
\item $\mu_{1} = \mu$
\item $\mu_{t} \boxplus \mu_{s} = \mu_{t+s}$
\end{enumerate}
The existence of this semigroup for each $\mu$ was proven in \cite{NS1}.  The case where the family extends to an $\mathbb{R}^{+}$ semigroup $\{ \mu_{t} \}_{t \geq 0}$ with $\mu_{0} = \delta_{0}$ are the \textit{free L\`{e}vy processes} \cite{Bi}.

Our main theorem is the following:
\begin{theorem}\label{MainT}
Let $\mu$ denote a compactly supported Borel probability measure on $\mathbb{R}$.
Then the family of sets $\{ supp(\mu_{t}) \}_{t > 1}$ is continuous in the Hausdorff metric.
\end{theorem}
We note that this theorem extends to  L\`{e}vy processes for $t >0$ through translation of time $t$.

\section{Preliminaries}
We assume that $\mu$ is not a Dirac mass.
The \textit{Cauchy transform} of $\mu$ is the function
$$G_{\mu}(z) := \int_{\mathbb{R}} \frac{1}{z-s} d\mu(s) : \mathbb{C}^{+} \mapsto  \mathbb{C}^{-}.$$
The reciprocal of this function is the \textit{F-transform}
$$ F_{\mu}(z) := \frac{1}{G_{\mu}(z)}:  \mathbb{C}^{+} \mapsto  \mathbb{C}^{+}.  $$
The main tool in our study is the boundary behavior of the $F$-transforms.

Part \eqref{BelT1} in the following theorem was first proven in \cite{BB1}.
\begin{theorem}[\cite{BelT}, Theorem 2.3]\label{BelT}
Let $t >1$.
\begin{enumerate}
\item\label{BelT1} A point $x \in \mathbb{R}$ satisfies $F_{\mu_{t}}(x) = 0$ if an only if $x/t$ is an atom of $\mu$ with mass
$$\mu(\{ x/t \}) \geq (t-1)/t.$$  If the inequality is strict, then $x$ is an atom of $\mu_{t}$, and 
$$\mu_{t}(\{ x \}) = t\mu \left( \left\{ \frac{x}{t} \right\}\right) - (t-1). $$
\item\label{BelT2} The nonatomic part of $\mu_{t}$ is absolutely continuous, and its density is continuous except at the (finitely many) points $x$ such that $F_{\mu_{t}}(x) = 0$.
\item\label{BelT3} The density of $\mu_{t}$ is analytic at all points where it is different from zero.
\end{enumerate}
\end{theorem}
We also note that the $F$-transform has continuous extension to $\overline{\mathbb{C}^{+}}$ for $t > 1$.

Consider the Nevanlinna representation
\begin{equation}\label{Nevanlinna}
 F_{\mu}(z) = \alpha + z + \int_{\mathbb{R}} \frac{1 + sz}{s - z}d\rho(s).
\end{equation}
The following methods for studying $\{ \mu_{t} \}_{t\geq 1}$ were developed in \cite{haowei}.  Let $$ g(x) := \int_{\mathbb{R}} \frac{s^{2} + 1}{(x-s)^{2}}d\rho(s) : \mathbb{R} \mapsto \mathbb{R}^{+} \cup \{ \infty \}. $$
For $t > 1$, we define
$$V_{t}^{+} := \left\{x \in \mathbb{R} : g(x) > \frac{1}{t-1} \right\} $$
$$  f_{t}(x) := \inf \left\{ y: \int_{\mathbb{R}} \frac{s^{2} + 1}{(x-s)^{2} + y^{2}}d\rho(s) \leq \frac{1}{t-1}    \right\}$$
and note that $f_{t}(x) > 0$ if and only if $x \in V_{t}^{+}$.
Lastly, we define
$$ H_{t}(z) :=  tz - (t-1)F_{\mu}(z), \  \  \ \psi_{t}(x) := H_{t}(x + if_{t}(x)).$$
In the remarks following Corollary 3.6 in \cite{haowei}, we have that
\begin{equation}\label{lipschitz}
|H_{t}(z_{1}) - H_{t}(z_{2})| \leq 2 |z_{1} - z_{2}|
\end{equation}
provided that $\Im{(z_{i})} \geq f_{t}(\Re{(z_{i})})$ (the original proof of this fact may be found in \cite{BBSG}).

The following theorem will serve as our main tool for the study of the absolutely continuous portion of the supports.  It was proven by Haowei Huang but many of the ideas were first formulated in Philippe Biane's paper \cite{phil}.
\begin{theorem}[\cite{haowei}, Theorem 3.8]\label{ACT}
Let $t > 1$.
 Then the  absolutely continuous part of $\mu_{t}$ is concentrated on the closure of  $\psi_{t}(V_{t}^{+})$.
\end{theorem}

\section{Main Results}
We assume that $\mu$ is compactly supported. 
Since all of our results are for $t > 1$, we may translate the semigroup.
Therefore, by \eqref{BelT}\eqref{BelT2}, we may assume that $\mu$ is absolutely continuous apart from a finite number of atoms.  We may also assume that $F_{\mu}$ extends continuously to $\overline{\mathbb{C}^{+}}$.
The proof of our main theorem is based on the following propositions.
\begin{proposition}\label{prop1}
Let $1 \leq \alpha < \beta$.
The set $V_{t}^{+}$ is
uniformly bounded and increasing for all $t \in (\alpha , \beta)$
 and is
Hausdorff continuous in the parameter $t$.
\end{proposition}
\begin{proof}

For $r < t$, $$\frac{1}{t-1} < \frac{1}{r-1}$$
so that $x \in V_{r}^{+}$ implies that $x \in V_{t}^{+}$.
Uniform boundedness follows from  the fact that $g(x) \rightarrow 0$ as $|x| \uparrow \infty$.

Observe that $$ \frac{d^{2}}{dx^{2}}g(x) = \int_{\mathbb{R}} \frac{6(s^{2} + 1)}{(x-s)^{4} }d\rho (s)$$
so that $$\frac{d^{2}}{dx^{2}}g(x) \in (0, \infty]$$ for all $x \in \mathbb{R}$.
We may conclude that $g$ is never constant on an interval and has no local maxima outside of the atoms, where $g(x) = \infty$.

Assume that $t_{n} \rightarrow t$ and   for some $\epsilon > 0$, $V^{+}_{t_{n}} \not\subseteq B_{\epsilon}(V_{t}^{+})$, the open unit ball, for all $n$.
Let $N > 0$ satisfy $V_{t_{n}}^{+} \subset [-N,N]$ for all $n$.
By compactness of $$[-N,N] \setminus B_{\epsilon}(V_{t}^{+}),$$ we have that $y_{n}$ subconverges to some $y \notin B_{\epsilon}(V_{t}^{+})$
  and satisfies $$g(y) = \frac{1}{t-1}.$$
The fact that $y \notin B_{\epsilon}(V_{t}^{+})$ implies that
$$ g(x) \leq \frac{1}{t-1}$$
for all $x\in B_{\epsilon}(y)$.  In particular, $g$ is either flat on some interval
containing $y$ or it has a local maxima at this point, providing a contradiction.

Assume, instead, that $V^{+}_{t} \not\subseteq B_{\epsilon}(V_{t_{n}}^{+})$.
By compactness, we may assume that there exists a $y \in V_{t}^{+}$ such that
$ y \notin B_{\epsilon/2}(V_{t_{n}}^{+})$.  This implies that, for every $w \in (y-\epsilon/2 , y + \epsilon/2)$,
$$ \frac{1}{t-1} \leq  g(w) \leq \frac{1}{t_{n}-1}.  $$
Letting $n\uparrow \infty$, this produces an interval where $g$ is constant, providing a contradiction.

\end{proof}

\begin{proposition}\label{ACmain}
The graphs $\{ (x,f_{t}(x))  :  x\in V_{t}^{+} \}$ are continuous in the Hausdorff metric for $t \in (1,\infty)$.
\end{proposition}
\begin{proof} 
We separate our proof into cases.  Fix $c > 0$ and $\epsilon > 0$.
Define $$ V_{t,c}^{+} := \{x \in V_{t}^{+} : f_{t}(x) \geq c \}. $$
We first assume that $t < r$ and claim that there exists $\delta > 0$ such that $r-t < \delta$ implies
\begin{equation}\label{claim1a}
f_{t}(x) < f_{r}(x) \leq (1+\epsilon) f_{t}(x)
\end{equation}
for all $x \in V_{t,c}^{+}$.

The first of these inequalities is obvious.  For the second, observe that
\begin{align*} \int_{\mathbb{R}} & \frac{1 + s^{2}}{(x - s^{2}) + (1 + \epsilon)^{2} f_{t}(x)^{2} }d\rho(s) \\
& \leq \sup_{x \in supp(\rho)} \left( \frac{(x - s^{2}) + f_{t}(x)^{2} }{(x - s^{2}) + (1 + \epsilon)^{2} f_{t}(x)^{2} } \right) \int_{\mathbb{R}} \frac{1 + s^{2}}{(x - s^{2}) +  f_{t}(x)^{2} }d\rho(s)
\\ &=  \sup_{x \in supp(\rho)} \left( \frac{(x - s^{2}) + f_{t}(x)^{2} }{(x - s^{2}) + (1 + \epsilon)^{2} f_{t}(x)^{2} } \right)  \frac{1}{t-1}
 \end{align*}
 Since $supp(\rho)$  is compact, 
the supremum is bounded below $1$ uniformly for those $f_{t}(x) \geq c$.  Thus, for $\delta$ small enough,
$$  \sup_{x \in supp(\rho)} \left( \frac{(x - s^{2}) + f_{t}(x)^{2} }{(x - s^{2}) + (1 + \epsilon)^{2} f_{t}(x)^{2} } \right) < \frac{t-1}{r-1} < 1. $$  This implies
$$  \int_{\mathbb{R}} \frac{1 + s^{2}}{(x - s^{2}) + (1 + \epsilon)^{2} f_{t}(x)^{2} }d\rho(s) < \frac{1}{r-1}$$
Since $f_{r}(x)$ is the infimum among positive real numbers satisfying this inequality, \eqref{claim1a} holds.

Assume that $r < t$ and $\epsilon > 0$.
We claim that there exists $\delta > 0$ such that $t-r < \delta$ implies
\begin{equation}\label{claim1b}
 f_{t}(x) (1-\epsilon) \leq f_{r}(x) < f_{t}(x)
\end{equation}
for all $x \in V_{t,c}^{+}$.
The second inequality is obvious.
We assume, for the sake of contradiction that $r_{n} \uparrow t$  and that $$ c \leq f_{t}(x_{n})  \ ; \ \  f_{r_{n}}(x_{n})< c(1-\epsilon).$$

By compactness, we may assume that $x_{n} \rightarrow x$ and $ c \leq f_{t}(x) $.
Thus,
\begin{align}
\frac{1}{t-1} &= \int_{\mathbb{R}} \frac{1 + s^{2}}{(x-s)^{2} + f_{t}(x)^{2} }d\rho(s) \label{ineqa} \\
&= \int_{\mathbb{R}} \frac{1 + s^{2}}{(x_{n}-s)^{2} + f_{r_{n}}(x_{n})^{2} } \frac{(x_{n}-s)^{2} + f_{r_{n}}(x_{n})^{2}}{(x-s)^{2} + f_{t}(x)^{2} } d\rho(s) \label{ineqb} \\
&\geq \inf_{s\in supp(\rho)} \frac{(x_{n}-s)^{2} + f_{r_{n}}(x_{n})^{2}}{(x-s)^{2} + f_{t}(x)^{2} } \int_{\mathbb{R}} \frac{1 + s^{2}}{(x_{n}-s)^{2} + f_{r_{n}}(x_{n})^{2} }  d\rho(s) \label{ineqc} \\
&= \frac{1}{r_{n} - 1} \inf_{s\in supp(\rho)} \frac{(x_{n}-s)^{2} + f_{r_{n}}(x_{n})^{2}}{(x-s)^{2} + f_{t}(x)^{2} } \label{ineqd}  > \frac{1}{t-1}
\end{align}
where the last inequality holds for $n$ large enough since the ratio is bounded strictly below $1$ by compactness of $supp(\rho)$.  This contradiction proves \eqref{claim1b}.

Note that $f_{t}(x) > 0$ if $\rho$ has an atom at $x$.
Pick $c$ small enough so that $f_{t}(x) > 2c$ at every atom.
Let $N > 0$ so that $supp(\rho) \subset [-N,N]$ and $V_{t}^{+} \subset \subset [-N,N].$

Consider the map $$ h_{x}(y):= \int_{\mathbb{R}} \frac{1 + s^{2} }{(x - s)^{2} + y^{2} } d\rho(s) :  [f_{t}(x) , \infty) \mapsto \left[\frac{1}{t-1},\infty\right).$$
We claim that $h_{x}$ is  continuous  at $f_{t}(x)$ for all $x \in [-N,N]\setminus V_{t,c}^{+}$, and that the rate of convergence is uniform.  Since $f_{t}(x) \leq y$, we have 
\begin{align}
|h_{x}(y) & - h_{x}(f_{t}(x))| = \left|\int_{\mathbb{R}}  \frac{1 + s^{2}}{(x - s)^{2} + f_{t}(x)^{2}} -   \frac{1 + s^{2}}{(x - s)^{2} + y^{2}}d\rho(s) \right| \\& = \int_{\mathbb{R}} \frac{(1 + s^{2})(y^{2} - f_{t}(x)^{2})}{[(x - s)^{2} + f_{t}(x)^{2}][(x - s)^{2} + y^{2}]}d\rho(s) \\
& = \int_{\mathbb{R}} \frac{(1 + s^{2})(y^{2} - f_{t}(x)^{2})}{[(x - s)^{2} + f_{t}(x)^{2}]^{2} + [(x - s)^{2} + f_{t}(x)^{2}][y^{2} - f_{t}(x)^{2}]}d\rho(s) \\
& \leq \int_{[-N , N] \setminus (x - \sqrt{y} , x + \sqrt{y})}\frac{(1 + s^{2})(y^{2} - f_{t}(x)^{2})}{[(x - s)^{2} + f_{t}(x)^{2}]^{2}}d\rho(s) \label{inta} \\
&  + \int_{(x - \sqrt{y} , x + \sqrt{y})} \frac{(1 + s^{2})(y^{2} - f_{t}(x)^{2})}{[(x - s)^{2} + f_{t}(x)^{2}][y^{2} - f_{t}(x)^{2} ]}d\rho(s) \label{intab}.
\end{align}

Pick $\epsilon > 0$.  We rewrite  \eqref{inta} as
\begin{equation}\label{intc}
 \int_{[-N , N] \setminus (x - \sqrt{y} , x + \sqrt{y})}\frac{(1 + s^{2})}{[(x - s)^{2} + f_{t}(x)^{2}]}\frac{(y^{2} - f_{t}(x)^{2})}{[(x - s)^{2} + f_{t}(x)^{2}]}d\rho(s) 
\end{equation}

Assume that $y = f_{t}(x) + \gamma$ where $\gamma \in [ 0 , \epsilon)$.
Let $\epsilon' = 2\gamma f_{t}(x) + \gamma^{2}$ and note that $$ y = \sqrt{f_{t}(x)^{2} + \epsilon'} \ , \ \ \epsilon^{2} < \epsilon' < 2c\epsilon + \epsilon^{2} $$
so that
$$\frac{(y^{2} - f_{t}(x)^{2})}{[(x - s)^{2} + f_{t}(x)^{2}]} \leq \frac{y^{2} - f_{t}(x)^{2}}{y + f_{t}(x)^{2}}   \leq \frac{\epsilon'}{\sqrt{\epsilon'}} \leq \sqrt{2c\epsilon + \epsilon^{2}}$$
Thus, $y - f_{t}(x) < \epsilon$ implies that   \eqref{intc} is bounded by
$$ \sqrt{2\epsilon + \epsilon^{2}} \int_{[-N , N] \setminus (x - \sqrt{y} , x + \sqrt{y})}\frac{(1 + s^{2})}{[(x - s)^{2} + f_{t}(x)^{2}]}d\rho(s) \leq \frac{\sqrt{2\epsilon + \epsilon^{2}} }{t-1}.$$
We conclude that \eqref{inta} is $O(\sqrt{y - f_{t}(x) })$ uniformly in $x$.

We focus on \eqref{intab}.  After cancellation, we have that it is equal to
$$ \int_{(x - \sqrt{y} , x + \sqrt{y})} \frac{(1 + s^{2})}{[(x - s)^{2} + f_{t}(x)^{2}]}d\rho(s). $$
Thus, this is an integrable function taken over arbitrarily small intervals.  We claim that this also converges to $0$ uniformly over $x \in [-N,N]\setminus V_{t,c}^{+}.$

Assume, for the sake of contradiction, that  $x_{n} \in [-N,N]\setminus V_{t,c}^{+}$ and $\delta_{n} > 0$ such that $\delta_{n} \rightarrow 0$ and 
$$ \int_{(x_{n} - \delta_{n} , x_{n} + \delta_{n})} \frac{(1 + s^{2})}{[(x_{n} - s)^{2} + f_{t}(x_{n})^{2}]}d\rho(s) > \gamma > 0 . $$  Since $[-N,N]\setminus V_{t,c}^{+}$ is compact, we may assume that we have a cluster point $x$.

Observe that
\begin{equation}\label{dominator} \frac{(1 + s^{2})}{(x - s)^{2} + [f_{t}(x)/2]^{2}} \in L^{1}(\rho). \end{equation}
When $f_{t}(x) > 0$ this follows from boundedness of the function and when $f_{t}(x) = 0$ this follows from the definition of $f_{t}$.  This will be our dominating function for the dominated convergence theorem.

Fix $\delta > 0$.  We have that
\begin{align*}
\int_{[x-\delta , x + \delta]} \frac{(1 + s^{2})}{(x - s)^{2} + f_{t}(x)^{2}} d\rho(s) & = \lim_{n\uparrow\infty} \int_{[x-\delta , x + \delta]} \frac{(1 + s^{2})}{(x - s)^{2} + f_{t}(x_{n})^{2}} d\rho(s)\\
&= \lim_{n\uparrow\infty} \int_{[x_{n}-\delta , x_{n} + \delta]} \frac{(1 + r^{2})}{(x_{n} - r)^{2} + f_{t}(x_{n})^{2}} d\rho(r)\\
 & + \int_{[x_{n}-\delta , x_{n} + \delta]} \frac{ (2r(x_{n} - x) + (x_{n}-x)^{2})}{(x_{n} - r)^{2} + f_{t}(x_{n})^{2}} d\rho(r) \\
& \geq \lim_{n\uparrow\infty} \int_{[x_{n}-\delta_{n} , x_{n} + \delta_{n}]} \frac{(1 + r^{2})}{(x_{n} - r)^{2} + f_{t}(x_{n})^{2}} d\rho(r) + O(x_{n} - x)\\
&\geq \gamma/2
\end{align*}
by the dominated convergence theorem and change of variables with $r = x_{n} - x + s$.

Since $x$ is not an atom of $\rho$, by the dominated convergence theorem,
$$ 0 = \lim_{\delta \rightarrow 0} \int_{[x-\delta , x + \delta]} \frac{(1 + s^{2})}{(x - s)^{2} + f_{t}(x)^{2}} d\rho(s) \geq \gamma/2 > 0,$$
providing our contradiction.  We conclude that \eqref{intab} converges to $0$ uniformly on $$[-N,N]\setminus V_{t,c}^{+},$$
proving that $h_{x}$ continuous at $f_{t}(x)$ with uniformity in the $x$ variable in the sense that
\begin{equation}\label{uniformity}
\lim_{\zeta \downarrow 0} \sup_{x \in [-N,N]\setminus V_{t,c}}|h_{x}(f_{t}(x) + \zeta ) - h_{x}(f_{t}(x))| = 0  
\end{equation}

To compliete the proof, let $\epsilon > 0$.  We show that there exists a $\delta > 0$ such that $|t-r|<\delta$ implies that
\begin{equation}\label{HD1}
\{ (y,f_{r}(y) \}_{y \in V_{r}^{+}} \subset B_{\epsilon}\left( \{ (x,f_{t}(x) \}_{x \in V_{t}^{+}} \right)
\end{equation}
\begin{equation}\label{HD2}
 \{ (x,f_{t}(x) \}_{x \in V_{t}^{+}}\subset B_{\epsilon}\left( \{ (y,f_{r}(y) \}_{y \in V_{r}^{+}}   \right)
\end{equation}

For \eqref{HD1}, let $x \in V_{t, c}^{+}$ with $c = \epsilon/2$.  Inequalities \eqref{claim1a} and \eqref{claim1b} imply one of the following
\begin{equation}\label{inter1}
0 \leq f_{r}(x) - f_{t}(x) \leq \epsilon' f_{t}(x) \leq \epsilon
\end{equation}
\begin{equation}\label{inter2}
-\epsilon \leq -\epsilon' f_{t}(x) \leq f_{r}(x) - f_{t}(x) \leq 0
\end{equation}
where $ f_{t}(x) < m$ for $x \in [-N,N]$ and $\epsilon' < \epsilon/m$.

We focus on the set $$[-N,N] \setminus V_{t,c}^{+}.$$  For $r < t$, , recall that $V_{r}^{+} \subset V_{t}^{+}$.  For $y \in V_{r}^{+} $ $$0 \leq f_{r}(y) < f_{t}(y) < \epsilon/2$$
so that $$ |(y,f_{t}(y) - (y , f_{r}(y))| < \epsilon/2. $$

For $r > t$,
 \eqref{uniformity} implies that for $\delta$ small, 
$$y \in [-N,N] \setminus V_{t,c}^{+} \ \Rightarrow \ f_{r}(y) < \epsilon/\sqrt{2}.$$
By \eqref{prop1} (for possibly smaller $\delta$), for any $y \in V_{r}^{+}\setminus V_{t,c}^{+}$, there exists an $x \in V_{t}^{+}$ such that $|x-y| < \epsilon/2$.
By continuity of $f_{t}$ (IVT), we may also assume that $f_{t}(x)<\epsilon/2$.  

  Thus,
 \begin{equation}\label{inter3} |(x,f_{t}(x)) - (y,f_{t}(y))|^{2} \leq  |x-y|^{2} + |f_{t}(x)|^{2} + |f_{r}(y)|^{2} \leq \frac{\epsilon^{2}}{4} + \frac{\epsilon^{2}}{4} + \frac{\epsilon^{2}}{2} = \epsilon^{2}.\end{equation}  This proves \eqref{HD1}.

We turn to  \eqref{HD2}. Note that  \eqref{inter1} and \eqref{inter2} hold for $\delta$ small  and $x \in V_{t, c}^{+}$.

For $x \in V_{t}\setminus V_{t, c}^{+}$ and $t < r$, note that $V_{t}^{+} \subset V_{r}^{+}$.   \eqref{uniformity} implies that $f_{r}(x) < \epsilon/\sqrt{2}$ for $\delta$ small, independent of $x$.  For $r < t$, by \eqref{prop1} there exists $y \in V_{r}^{+}$ such that $|x-y|< \epsilon/2$, for $\delta$ small.  By continuity of $f_{r}$, we may assume that $f_{r}(y) < \epsilon/2$.  The inequalities in  \eqref{inter3} hold once again.

This completes our proof.
\end{proof}

\begin{proof}[Proof of Theorem \eqref{MainT}]
Invoking \eqref{lipschitz}, for $r < t$ we have
\begin{align}
| H_{t}(x +& if_{t}(x)) - H_{r}(y + if_{t}(y)) |  \\
&\leq \left| H_{t}(x + if_{t}(x)) - H_{r}(x + if_{t}(x)) \right|  + \left| H_{r}(x + if_{t}(x)) - H_{r}(y + if_{r}(y)) \right| \\
&\leq (t-r)\left[ |x + if_{t}(x)| + |F_{\mu}(x + if_{t}(x))|\right] + 2 |(x,f_{t}(x)) - (y,f_{r}(y))| \\
&= \label{INEQ1} O\left(\min\left\{(t-r) , |(x,f_{t}(x)) - (y, f_{r}(y))|\right\} \right)
\end{align}
and for $t < r$ we have\begin{align}
| H_{t}(&x + if_{t}(x)) - H_{r}(y + if_{t}(y)) |  \\
&\leq (r-t)\left[ |y + if_{r}(y)| + |F_{\mu}(y + if_{r}(y))|\right] + 2 |(x,f_{t}(x)) - (y,f_{r}(y))| \\ 
&= \label{INEQ2} O\left(\min\left\{(r-t) , |(x,f_{t}(x)) - (y, f_{r}(y))|\right\}\right)
\end{align}
since we may bound $f_{r}$ uniformly over $[-N,N]$ and $r \in (t-\delta , t + \delta)$ by Proposition \eqref{ACmain}
and $F_{\mu}$ is continuous on this compact set.

Thus, for any point $H_{t}(x + if_{t}(x))$ with $x \in V_{t}^{+}$, for $|t-r|$ small enough, there exists $y\in V_{r}^{+}$
arising from \eqref{ACmain} such that \eqref{INEQ1} is smaller than $\epsilon$ for $r < t$ and \eqref{INEQ2} is smaller than $\epsilon$ for $t < r$. 
By \eqref{ACT},
$$ supp(\mu_{r}^{ac})= \overline{\psi_{r}(V_{r}^{+})}  = \overline{\{ H_{r}(y + i f_{r}(y) \}_{r \in V_{r}^{+}}} \subset B_{\epsilon}(supp(\mu_{t}^{ac})). $$
The reverse inclusion follows similarly so that $\mu_{t}^{ac}$ has Hausdorff continuous support.  

By \eqref{BelT}, 
\begin{equation}\label{atoms}
\mu_{t}(\{ \alpha \}) = t\mu\left(\left\{ \frac{\alpha}{t} \right\}\right) - (t-1).
\end{equation}
so that the atoms vary continuously.  

The last remaining case are those time $t$ such that
$$ t\mu\left(\left\{ \frac{\alpha}{t} \right\}\right) - (t-1) = 0. $$
That is, at those times $t$ where an atom vanishes.  We claim that, for such a $t$,
\begin{equation}\label{mainatomclaim}
\alpha \in \psi_{t}\left(\overline{V_{t}^{+}}\right) =   \overline{\psi_{t}(V_{t}^{+})} = supp(\mu_{t}^{ac}).
\end{equation}

We assume without loss of generality that $\alpha = 0$ (as translation commutes with the semigroup operation).
Let $$  r\mu\left(\left\{ 0 \right\}\right) - (r-1) = 0 \  \Rightarrow \ \mu\left(\left\{ 0 \right\}\right) = \frac{r-1}{r}$$
for $r > 1$.
Rewriting the Cauchy transform, we have that
$$ G_{\mu}(z) = \frac{r-1}{rz} + \frac{G_{\nu}(z)}{r} $$
where $\nu$ is the probability measure obtained by removing this mass from $\mu$ and normalizing.

We may assume that $$ \nu((-\eta , \eta)) = 0. $$
Indeed, if for every $\epsilon > 0$ there exists a $\delta > 0$ such that $0 \leq t-r<\delta$ implies that  $$\mu_{r}((-\epsilon,0) \cup (0,\epsilon))>0$$   then our theorem is true.  Indeed, since we have finitely many atoms, this implies that $$supp(\mu_{r}^{ac})\cap (-\epsilon,0) \cup (0,\epsilon) \neq \emptyset$$
and for all $t-r < \delta$ we have just proved Hausdorff continuity of $supp(\mu_{r}^{ac})$.

Using the Nevanlinna representation
$$ F_{\mu}(z) = \alpha + z + \int_{\mathbb{R}}\frac{1 + sz}{s-z}d\rho(s) $$
after taking the imaginary part of both sides and some mild algebra, we have that for $z =  iy$,
\begin{align}
\int_{\mathbb{R}} \frac{1 + s^{2}}{s^{2} + y^{2}} d\rho(s) & = \frac{\Im{(F_{\mu}(iy)})}{y} - 1 \label{Nequality1}  \\
&= \Im \left( \frac{i}{iy  G_{\mu}(iy)  } \right) - 1 \\
&= \Im \left( \frac{i}{iy\left( \frac{r-1}{riy} + \frac{G_{\nu}(iy)}{r}  \right)} \right) - 1 \label{Nequality2}
\end{align}
We consider the limit as $y \downarrow 0$. 
By Lemma 7.1 in \cite{BV1}, $$ \lim_{y \downarrow 0} \ iyG_{\nu}(iy) = 0.$$
We conclude that 
$$ \lim_{y \downarrow 0}  \int_{\mathbb{R}} \frac{1 + s^{2}}{s^{2} + y^{2}} d\rho(s) = \lim_{y \downarrow 0}  \Im \left( \frac{i}{iy\left( \frac{r-1}{riy} + \frac{G_{\nu}(iy)}{r}  \right)} \right) - 1  = \frac{r}{r-1} - 1 = \frac{1}{r-1}.$$
 so that $0 \in V_{r +  \epsilon}^{+}$ for all $\epsilon > 0$.

By Proposition \eqref{prop1}, $0 \in \overline{V_{r}^{+}}$ and $f_{r}(0) = 0$.
Thus, $H_{r}(0) \in supp(\mu_{r}^{ac})$.  But,
$$ H_{r}(0) =  0 - (r-1)F_{\mu}(0) = 0. $$
We conclude that  $0 \in supp(\mu_{r}^{ac})$, proving our theorem.

\end{proof}

\section{conclusion}
This significance of this work is another strong regularity result that is a characteristic of free random variables (see also \cite{BV2}).
This also has implications for recent quantum information results found in \cite{benoit} insofar as their entanglement test is now dependent on a continuous function and is therefore not prone to erratic behavior.  
We also note that these free L\`{e}vy processes are the limiting objects of certain processes with range in large matrices \cite{thierry, georges} and it would be interesting to see if some type of coarse continuity phenomenon may be found in the empirical eigenvalue distribution of these objects.

\section*{Acknowledgements}
I would like to thank Hao-Wei Huang for useful discussions.  I am grateful towards the Alexander von Humboldt-Stiftung for their generous support.

%To finish the proof, note that Proposition 3.5 in \cite{haowei} implies that
%$$ |F_{\mu}(z_{1}) - F_{\mu}(z_{2})| \leq \frac{t}{t-1}|z_{1} - z_{2}|  , \ z_{1} , z_{2} \in \overline{\Omega_{t}}.$$
%Thus, fix $\epsilon > 0$.  Pick $\delta > 0$ such that \eqref{mainclaim} holds with radius $\epsilon'$ so that the inequality 
%$$  \delta \epsilon' \left(1 + \min \left\{ \frac{t}{t-1}  , \frac{r}{r-1}\right\} \right) < \epsilon$$
%Then, for $x \in V_{t}^{+}$ there exists a $y \in V_{r}^{+}$ such that 
%\begin{align} |\psi_{t}(x) - \psi_{r}(y)|  &\leq |t-r| \left| [x + if_{t}(x) - y - if_{r}(y)] + F_{\mu}(x + if_{t}(x)) - F_{\mu}( y + if_{r}%(y)) \right| \\
%&\leq|t-r| \left( \epsilon' +   \min \left\{ \frac{t}{t-1}  , \frac{r}{r-1}\right\}\epsilon' \right) < \epsilon
% \end{align}
%We are able to conclude that $$ supp(\mu_{t})^{ac} = \overline{\psi_{t}(V_{t}^{+})} \subset B_{\epsilon}\left( \overline{\psi_{r}%(V_{r}^{+})} \right) = B_{\epsilon}(supp(\mu_{r})^{ac}) $$
%Moreover, for eve%ry $y \in V_{r}^{+}$ where $|t-r|<\delta$, there exists an $x \in V_{t}^{+}$ satisfying the same inequalities, so that we may conclude 
%$$ supp(\mu_{r})^{ac} = \overline{\psi_{r}(V_{r}^{+})} \subset B_{\epsilon}\left( \overline{\psi_{t}(V_{t}^{+})} \right) = %B_{\epsilon}(supp(\mu_{t})^{ac}) $$
%proving Hausdorff continuity.

\bibliographystyle{amsalpha}
\bibliography{Cont}
\end{document}